\documentclass[a4paper,11pt]{amsart}
\usepackage{amsmath}
\usepackage{amsthm}
\usepackage{amssymb}
\usepackage{amsfonts}
\usepackage{enumerate}
\usepackage[all]{xy}
\usepackage{graphicx}
\usepackage{multicol}
\usepackage{fancyhdr}
\usepackage{tikz}

\usepackage[utf8]{inputenc}

\normalsize

\newtheoremstyle{theoremstyle}
  {10pt}      
  {5pt}       
  {\itshape}  
  {}          
  {\bfseries} 
  {:}         
  {.5em}      
  {}          

\newtheoremstyle{examplestyle}
  {10pt}      
  {5pt}       
  {}          
  {}          
  {\bfseries} 
  {:}         
  {.5em}      
  {}          

\theoremstyle{theoremstyle}
\newtheorem{theorem}{Theorem}[section]
\newtheorem*{theorem*}{Theorem}
\newtheorem{lemma}[theorem]{Lemma}
\newtheorem*{lemma*}{Lemma}
\newtheorem{proposition}[theorem]{Proposition}
\newtheorem*{proposition*}{Proposition}
\newtheorem{corollary}[theorem]{Corollary}
\newtheorem*{corollary*}{Corollary}

\theoremstyle{examplestyle}

\newtheorem*{gp*}{Geometric Principle}

\newtheorem{definition*}{Definition}
\newtheorem{remark}[theorem]{Remark}
\newtheorem{remark*}{Remark}
\newtheorem{convention}[theorem]{Convention}
\newtheorem{convention*}{Convention}

\newtheorem{notation*}{Notation}

\newtheorem{stl*}{Structure Lemma}
\newtheorem*{tha*}{Theorem A}
\newtheorem*{thb*}{Theorem B}
\newtheorem*{thc*}{Theorem C}

\newcommand{\comment}[1]{}
\newcommand{\unaryminus}{\scalebox{0.75}[1.0]{\( - \)}}

\newcommand{\spec}[1]{\operatorname{spec}(#1)}

\newcommand{\Hom}{\operatorname{Hom}}

\newcommand{\End}{\operatorname{End}}

\newcommand{\rk}{\operatorname{rk}}

\newcommand{\Span}{\operatorname{span}}

\newcommand{\Mat}{\operatorname{Mat}}
\newcommand{\im}{\operatorname{im}}

\newcommand{\eins}{{\mathchoice
{\mathrm 1\mskip-4.2mu\mathrm l}{\mathrm 1\mskip-4.2mu\mathrm l}
{\mathrm 1\mskip-3.9mu\mathrm l}{\mathrm 1\mskip-4.0mu\mathrm l}}}

\newcommand{\N}{\mathbb{N}}
\newcommand{\Z}{\mathbb{Z}}
\newcommand{\R}{\mathbb{R}}
\newcommand{\C}{\mathbb{C}\mkern1mu}

\newcommand{\Sph}{\mathbb{S}}
\DeclareMathOperator{\Sym}{Sym}

\newcommand{\bmat}{\left(\begin{smallmatrix}}
\newcommand{\emat}{\end{smallmatrix}\right)}

\newcommand{\SO}{\mathrm{SO}}
\newcommand{\OO}{\mathrm{O}}

\newcommand{\U}{\mathrm{U}}

\newcommand{\Gtwo}{\mathrm{G}_2}

\begin{document}
\title[Classification of isoparametric hypersurfaces in spheres with $(g,m)=(6,1)$]{Classification of isoparametric hypersurfaces in spheres with $(g,m)=(6,1)$}
\author{Anna Siffert$^1$}%
\address{University of Pennsylvania\\
Department of Mathematics\\
Philadelphia PA 19104\\
USA}
\email{asiffert@math.upenn.edu}

\begin{abstract}
We classify the isospectral families $L(t)=\cos(t)L_0+\sin(t)L_1\in\mbox{Sym}(5,\R)$, $t\in\R$, with $L_0=\mbox{diag}(\sqrt{3},\tfrac{1}{\sqrt{3}},0,-\tfrac{1}{\sqrt{3}},-\sqrt{3})$.
Using this result we provide a classification of isoparametric hypersurfaces in spheres with $(g,m)=(6,1)$ and thereby 
give a simplified proof of the fact that any isoparametric hypersurface with $(g,m)=(6,1)$ is homogeneous.
This result was first proven by Dorfmeister and Neher \cite{dn}.
\end{abstract}

\maketitle

\footnotetext[1]{I would like to thank DFG for supporting this work with the grant SI 2077/1-1.}

\section*{Introduction}
\thispagestyle{empty}

The principal result of this paper is the following theorem.

\begin{theorem*}
Let $L(t)=\cos(t)L_0+\sin(t)L_1\in\mbox{Sym}(5,\R)$, $t\in\R$, be isospectral where $L_0=\mbox{diag}(\sqrt{3},\tfrac{1}{\sqrt{3}},0,-\tfrac{1}{\sqrt{3}},-\sqrt{3})$. Up to conjugation by an element $A\in\OO(5)$ with $AL_0A^{-1}=L_0$, the matrix $L_1$
is given by one of the following matrices
\begin{align*}
&L_1=\tfrac{1}{3\sqrt{3}}\left(\begin{smallmatrix}
0&5&0&2&0\\
5&0&4&0&2\\
0&4&0&4&0\\
2&0&4&0&5\\
0&2&0&5&0
\end{smallmatrix} \right),\,L_1=\tfrac{1}{\sqrt{6}}\left(\begin{smallmatrix}
0&0&0&0&3\sqrt{2}\\
0&0&1&0&0\\
0&1&0&1&0\\
0&0&1&0&0\\
3\sqrt{2}&0&0&0&0
\end{smallmatrix} \right),\,L_1=\tfrac{1}{\sqrt{6}}\left(\begin{smallmatrix}
0&0&3&0&0\\
0&0&0&\sqrt{2}&0\\
3&0&0&0&3\\
0&\sqrt{2}&0&0&0\\
0&0&3&0&0
\end{smallmatrix} \right),\\
&L_1=\tfrac{1}{\sqrt{3}}\left(\begin{smallmatrix}
0&1&0&2&0\\
1&0&0&0&2\\
0&0&0&0&0\\
2&0&0&0&1\\
0&2&0&1&0
\end{smallmatrix} \right),\,L_1=\tfrac{1}{\sqrt{3}}\left(\begin{smallmatrix}
0&\sqrt{3}&0&0&0\\
\sqrt{3}&0&0&2&0\\
0&0&0&0&0\\
0&2&0&0&\sqrt{3}\\
0&0&0&\sqrt{3}&0
\end{smallmatrix} \right),\,L_1=\tfrac{1}{\sqrt{3}}\left(\begin{smallmatrix}
0&0&0&0&3\\
0&0&0&1&0\\
0&0&0&0&0\\
0&1&0&0&0\\
3&0&0&0&0
\end{smallmatrix} \right).
\end{align*}
\end{theorem*}

Using this result we classify isoparametric hypersurfaces in spheres with six different principal curvatures $g=6$ all of multiplicity $m=1$ and thereby give a simplified proof of a result of Dorfmeister and Neher \cite{dn}.

\smallskip

In \cite{mi1,mi3} Miyaoka claims to reprove the result of Dorfmeister and Neher. Based on the idea of \cite{mi1,mi3} Miyaoka \cite{mi2} proposed how to establish homogeneity for  isoparametric hypersurfaces in spheres with six different principal curvatures $g=6$ all of multiplicity $m=2$, which is the only remaining open case with $g=6$.
Using (parts of) our main result we give a counterexample to Miyaoka's proof \cite{mi1,mi3}.

\smallskip

The present paper is organized as follows: 
the above theorem is proved in Section\,\ref{sec3} and used in Section\,\ref{sec4} to classify isoparametric hypersurfaces in $\Sph^7$ with $g=6$. Finally, the counterexample to the proof of Miyaoka \cite{mi1,mi3} can be found in the Appendix.

\section{Classification of the isospectral families }
\label{sec3}
Subsections\,\ref{11}-\ref{app} of this section serve as preparation for Subsection\,\ref{ch3} in which we prove the theorem stated in the introduction.

\subsection{Minimal polynomial equation}
\label{11}
In what follows we consider $L(t)\in\mbox{Sym}(5,\R)$, $t\in\R$, with
$$\mbox{spec}(L(t))=\left\{-\sqrt{3},-\tfrac{1}{\sqrt{3}},0,\tfrac{1}{\sqrt{3}},\sqrt{3}\right\}\,\,\,\mbox{for all}\,\,t\in\R,$$
where the eigenvalues arise with multiplicity $m$. Below we use the short hand notation $L$ for $L(t)$.
Thus we obtain the minimal polynomial equation
\begin{align*}
0=(L^2-3\cdot\eins)\,(L^2-\tfrac{1}{3}\cdot\eins)\,L=(L^4-\tfrac{10}{3}\,L^2+\eins)\,L.
\end{align*}
We introduce the complexified operators $L_{\pm}\in\End(\R^{5m}\otimes\C)$ by $L_{\pm}=\tfrac{1}{2}\,(L_0\mp i\,L_1).$
Since $L_0,L_1\in\End \R^{5m}$ are symmetric, $L_+,L_-\in\End(\R^{5m}\otimes\C)$ are also symmetric.
Plugging $L(t)=\exp(it)\,L_{+}+\exp(-it)\,L_{-}$ in the above equation and sorting by different frequencies yields
\begin{alignat}{2}
 & L_+^5=0,&&L_-^5=0,\\
&15\,\sigma(L_+^4\,L_-)-10\,L_+^3=0,&&15\,\sigma(L_+\,L_-^4)-10\,L_-^3=0,\\
&10\,\sigma(L_+^3\,L_-^2)-10\,\sigma(L_+^2\,L_-)+L_+=0,\hspace{0.6cm}&&10\,\sigma(L_+^2\,L_-^3)-10\,\sigma(L_+\,L_-^2)+L_-=0,
\end{alignat}
where $\sigma(L_+^i\,L_-^j)\in\Sym(\R^{5m}\otimes\C)$ is given by the sum of all possible words of $L_+^i\,L_-^j$ divided by the number of possible words, for example
\begin{align*}
\sigma(L_+^3\,L_-)=\tfrac{1}{4}\,(L_+^3\,L_-+L_+^2\,L_-\,L_++L_+\,L_-\,L_+^2+L_-\,L_+^3).
\end{align*}
It suffices to consider the first equation in each of the above rows, since the remaining equations are obtained from these by complex conjugation.

\subsection{The projector onto the kernel of $L(t)$}
\begin{lemma}
\label{projector}
For $t\in\R$ the map $P(t):\R^{5m}\rightarrow \R^{5m}$ given by $P(t)=L(t)^4-\tfrac{10}{3}\,L(t)^2+\eins$
is the projector onto the $m$-dimensional kernel of $L(t)$ .
\end{lemma}
\begin{proof}
Below we use the short hand notation $P=P(t)$.

\smallskip

On the one hand we have $L\,P=P\,L=0$ by the minimal polynomial equation, i.e., $\im P\subset\ker L$.
On the other hand, $x\in\ker L$ implies $P\,x=x$, i.e., $\ker L\subset\im P$.
Consequently, $\im P=\ker L$.
Finally, 
\begin{align*}
P^2-P=(P-\eins)P=(L^4-\tfrac{10}{3}L^2)P=(L^3-\tfrac{10}{3}L)LP=0,
\end{align*}
i.e., $P(t)$ is a projector for all $t\in\R$.
\end{proof}

Substituting $L(t)=\exp(it)\,L_++\exp(-it)\,L_-$ in the formula for $P(t)$ yields
\begin{align*}
P(t)=\exp(4it)\,P_4+\exp(2it)\,P_2+P_0+\exp(-2it)\,P_{-2}+\exp(-4it)\,P_{-4},
\end{align*}
where $P_4,P_2,P_0,P_{-2},P_{-4}\in\Sym(\R^{5m}\otimes\C)$ are given by
\begin{align*}
P_4=L_+^4,\hspace{0.5cm}P_{-4}=\overline{P_4},\hspace{0.5cm} P_2= 4\,\sigma(L_+^3\,L_-)-\tfrac{10}{3}\,L_+^2,\hspace{0.5cm}P_{-2}=\overline{P_2}
\end{align*}
and $ P_0= 6\,\sigma(L_+^2\,L_-^2)-\tfrac{20}{3}\,\sigma(L_+\,L_-)+\eins$. Clearly, $P_0=\overline{P_0}.$

\begin{lemma}
\label{cmp}
The minimal polynomial equation is equivalent to
\begin{align*} 
L_+\,P_4=0,\,\,\,\,\,\,\,\,\,\,L_{+}\,P_2+L_{-}\,P_4=0,\,\,\,\,\,\,\,\,\,\,L_{+}\,P_0+L_{-}\,P_2=0.
\end{align*}
\end{lemma}

\begin{corollary}
\label{ortlem}
$P_i\,L_{\pm}\,P_j=0$ for all $i,j\in I:=\left\{-4,-2,0,2,4\right\}.$
\end{corollary}

\begin{proof}
We have to establish $5\times 2\times 5=50$ equations. Obviously, given one equation, the transposed and the conjugate equation are also true, which has to be considered when counting equations.
Applying $P_i$ with $i\in I$ from the left to $L_+\,P_4=0$ we obtain $P_j\,L_+\,P_4=0$ for $j\in I$.
These are $18$ equations.
Using this result and Lemma\,\ref{cmp} we get
\begin{align*}
&P_4\,L_-\,P_4=P_4\,(-L_+\,P_2)=-(L_+\,P_4)\,P_2=0, P_4\,L_-\,P_2=P_4\,(-L_+\,P_0)=0,\\
&P_{-4}\,L_+\,P_2=P_{-4}\,(-L_-\,P_4)=0, P_{-4}\,L_+\,P_0=P_{-4}\,(-L_-\,P_2)=0.
\end{align*}
Hence, we proved $2+4+4+4=14$ additional equations. These identities again together with the identity $L_{+}\,P_2+L_{-}\,P_4=0$ of Lemma\,\ref{cmp} imply
$P_0\,L_+\,P_2=P_0\,(-L_-\,P_4)=0$ and similarly $P_2\,L_+\,P_2=0$ and $P_{-2}\,L_+\,P_{2}=0$, which are $4+2+4=10$ additional equations.
Combining these identities with Lemma\,\ref{cmp} yields $P_2\,L_-\,P_2=P_2\,(-L_+\,P_0)=0$ and $P_{-2}\,L_+\,P_0=P_{-2}\,(-L_-\,P_2)=0$, which are $2+4=6$ additional equations. The two remaining equations, $P_0\,L_{\pm}\,P_0=0$, are obtained by combining $P_{-2}\,L_+\,P_0=0$ and $L_{+}\,P_0+L_{-}\,P_2=0$.
\end{proof}

\subsection{The span of the kernel over time}
Following Miyaoka \cite{mi1} we introduce
\begin{align*}
E={\Span}_{t\in\R}\ker L(t)\subset \R^{5m}.
\end{align*}
Obviously, the independence of $\ker L(t)$ of $t\in\R$ is equivalent to $\dim E=m$. 

\begin{lemma}
\label{ehut}
\label{dr}
$E=\sum_{i\in I}\im P_i$ and $\dim E\leq 3m$.
\end{lemma}
\begin{proof}
Since $\im P(t)=\ker L(t)$ we have to prove ${\Span}_{t\in\R}\im P(t)=\sum_{i\in I}\im P_i$.
Clearly,
${\Span}_{t\in\R}\im P(t)\subseteq\sum_{i\in I}\im P_i.$
Hence the first claim follows from the identities
\begin{align}
\label{1}
\exp(4it)P_4+P_0+\exp(-4it)P_{-4}&=\tfrac{1}{2}(P(t)+P(t+\tfrac{\pi}{2})),\\
\label{2}P_0&=\tfrac{1}{3}(P(t)+P(t+\tfrac{\pi}{3})+P(t+\tfrac{2\pi}{3})),\\
\label{3}\exp(2it)P_2+\exp(-2it)P_{-2}&=\tfrac{1}{2}(P(t)-P(t+\tfrac{\pi}{2})).
\end{align}
In order to prove the second claim let $d=\dim E$.
Using $\dim(\ker L(t))=m$ for $t\in\R,$ we get $\dim(L(t)E)\geq d-m.$
Corollary\,\ref{ortlem} implies $L(t)\,E\perp E$ for all $t\in\R$ and thus $L_{\pm}\,E\perp E.$
Combining $L(t)E\subset E^{\perp}$ and $\dim(L(t)E)\geq d-m$ we obtain
$\dim E^{\perp}\geq d-m.$
From $E\oplus E^{\perp}=\R^{5m}$ we have 
$\dim E+\dim E^{\perp}=\dim \R^{5m}$. Thus we get
$5 m=\dim \R^{5m}=\dim E+\dim E^{\perp}\geq 2d-m,$
whence the claim.
\end{proof}

\begin{corollary}
$L(t)\,E\perp E$ for all $t\in\R$ and thus $L_{\pm}\,E\perp E.$
\end{corollary}

\begin{lemma}
The following five statements are equivalent: (i) $\ker L(t)$ is constant, (ii) $\dim E=m$, (iii) $L(t)\,E=0$ for $t\in\R$,
(iv) $L_+\,E=0$,	 (v) $L_+\,P_i=0$ for all $i\in\left\{-4,-2,0,2,4\right\}.$
\end{lemma}
\begin{proof}
The equivalence of (iv) and (v) follows from Lemma\,\ref{ehut}, the rest is obvious. 
\end{proof}

\begin{remark}
We will see below (see e.g. Lemma\,\ref{falle4nn}) that the minimal polynomial equation of one focal manifold is not sufficient to prove $\dim E=m$:
 we construct explicitly isospectral families which satisfy the minimal polynomial equation but
have a non-constant kernel. 
\end{remark}

\subsection{Some linear algebra}
\label{app}

In this subsection we provide some linear algebra results which we will need for the proofs in Subsection\,\ref{ch3}.

\smallskip

We denote by $\left\{e_1,e_2\right\}$ and $J=\bigl(\begin{smallmatrix}
0&-1\\ 1&0
\end{smallmatrix} \bigr)$ the standard basis of $\C^2$ and the usual almost complex structure of $\C^2$, respectively. Below we work with the basis $\left\{e_+,e_-\right\}$ of $\C^2$ built by the isotropic vectors
$e_{\pm}=\tfrac{1}{\sqrt{2}}\,(e_1\pm i\,e_2)$. A basis of $M_2(\C)$ is given by $\left\{\rho,\overline{\rho},\sigma,\overline{\sigma}\right\}$, where
$\rho=e_+\,e_-^{tr}=\tfrac{1}{2}\,(\eins+i\,J),\,\overline{\rho}=e_-\,e_+^{tr}=\tfrac{1}{2}\,(\eins-i\,J),\,\sigma=e_+\,e_+^{tr}$ and $\overline{\sigma}=e_-\,e_-^{tr}$.

\begin{lemma}
 The following identities hold:
\begin{enumerate}
\item $\rho^2=\rho,\,\overline{\rho}^2=\overline{\rho},\,\rho\overline{\rho}=0,\,\overline{\rho}\rho=0,\rho+\overline{\rho}=\eins$, $i J=\rho-\overline{\rho}, \,\rho^{tr}=\overline{\rho},\,e_+^{tr}\,\rho=0,\,e_-^{tr}\,\overline{\rho}=0$.
\item $\sigma^2=0,\,\overline{\sigma}^2=0,\,\sigma^{tr}=\sigma,\,\overline{\sigma}^{tr}=\overline{\sigma},\,e_+^{tr}\,\sigma=0,\,e_-^{tr}\,\overline{\sigma}=0.$
\item $\rho\,\sigma=\sigma=\sigma\,\overline{\rho},\hspace{0.2cm}\overline{\rho}\,\sigma=\sigma\,\rho=0,\hspace{0.2cm}\sigma\,\overline{\sigma}=\rho.$
\end{enumerate}
\end{lemma}

\begin{lemma}
For $B\in\Hom(\C^k,\C^{2\,l})$ the statement
$(e_+^{tr}\otimes\eins_l)B=0$ is equivalent to $B=e_+\otimes B_0\,\,\,\,\mbox{for some}\,\, B_0\in \Hom(\C^k,\C^l).$ Furthermore, $B_0$ is given by
$B_0=(e_-^{tr}\otimes\eins_l)B$ and is thus uniquely determined by $B.$
\end{lemma}

\begin{corollary}
For $B\in\Hom(\C^k,\C^{2\,l}), c\in\C^*$ and an injective $A_0\in\End\C^l$ with
$(c\,e_+^{tr}\otimes A_0)B=0$ we have $B=e_+\otimes B_0$ where $B_0=(e_-^{tr}\otimes\eins_l)B$.
\end{corollary}

Note that a change of the basis in $\OO_n(\R)$ is compatible with the structure of the problem: let $U\in\OO_n(\R)=\OO_n(\C)\cap\U(n)$ be given and set
$L_+^{'}=UL_+ U^{tr}$, and $L_-^{'}=UL_-U^{tr}$. Thus $L_{\pm}^{'}$ satisfy the same identities as $L_{\pm}$.

\begin{lemma}
\label{stl}
For $A\in\End\C^{d}$ with $A^{tr}=A$ and $A^2=0$ there exist $U\in\OO_{d}(\R)$ and a positive definite, diagonal matrix $A_0\in\End\R^{d_0}\subset\End\C^{d_0}$ such that
$UAU^{tr}=\bigl(\begin{smallmatrix}
\sigma\otimes A_0&0\\ 0&0
\end{smallmatrix} \bigr)$.
\end{lemma}
\begin{proof}
The real and symmetric matrices $\mbox{Re}(A)=\tfrac{1}{2}\,(A+A^c)$ and $\mbox{Im}(A)=\tfrac{1}{2i}\,(A-A^c)$
satisfy $\mbox{Re}{(A)}^2=\tfrac{1}{4}\,\left\{A,A^c\right\}=\mbox{Im}{(A)}^2$, where $A^c$ denotes the conjugate of $A$.
Consequently, $Q:=\tfrac{1}{4}\,\left\{A,A^c\right\}$
is a positive semi definite matrix and therefore $\ker Q^{\perp}=\im Q$.
Since $A$ and $A^c$ commute with $Q$, the endomorphism $A$ and $A^c$ map the subspace $\ker Q^{\perp}=\im Q$ onto itself. Moreover, using $\mbox{Re}{(A)}^2=\tfrac{1}{4}\,\left\{A,A^c\right\}=\mbox{Im}{(A)}^2$ we prove easily that $A$ and $A^c$ vanish on the subspace $\ker Q$. By a straightforward computation we verify that
 $J_0:=-\tfrac{1}{4}\,i\,(A\,Q^{-1}\,A^c-A^c\,Q^{-1}\,A)$
defines an almost complex structure on the subspace $\im Q$ and thus there exists a $d_0\in\N$ such that $\dim(\im Q)=2d_0$.
We can choose a basis of $\im Q$ such that $\mbox{Re}(A)$ is diagonal.
Moreover, we have $J_0\mbox{Re}(A)=\mbox{Im}(A)$ and $J_0\mbox{Im}(A)=-\mbox{Re}(A)$. 
Let $\mbox{Re}(A)=\mbox{diag}(A_1,A_2)$ where $A_1,A_2\in\mbox{diag}(d_0,\R)$. We thus get $A_2=-A_1$ and we can choose 
 a basis of $\im Q$ such that $\mbox{Re}(A)=\mbox{diag}(A_0,-A_0)$, where $A_0\in\mbox{diag}(d_0,\R)$ is positive definite.
\end{proof}
 
\begin{convention}
Let a symmetric matrix $A$ with $A^2=0$ be given.
Below we write for short that Lemma\,\ref{stl} implies that there exists 
a diagonal matrix $A_0$, which is positive definite or the null matrix such that $A=\bigl(\begin{smallmatrix}
\sigma\otimes A_0&0\\ 0&0
\end{smallmatrix} \bigr)$, i.e., we will not mention that this identity only holds up to conjugation by an element of the orthogonal group.
\end{convention}

\begin{lemma}
\label{multi}
For $A=\sigma\otimes A_0\in\Mat(2n_1,\C)\,\,\mbox{and} \,\,B=\sigma\otimes B_0\in\Mat(2n_2,\C)$, where $A_0\in\Mat(n_1,\R)$ and $B_0\in\Mat(n_2,\R)$ are positive definite, diagonal matrices, we have
\begin{enumerate}
\item $C A=0$ for $C\in\Mat(n_3\times2n_1,\C)\Rightarrow C=e_+^{tr}\otimes C_0$ with $C_0\in\Mat(n_3\times n_1,\C)$,
\item $B C=0$ for $C\in\Mat(2n_2\times n_4,\C)\Rightarrow C=e_+\otimes C_0$ with $C_0\in\Mat(n_2\times n_4,\C)$,
\item $C A=0$ and $B C=0$ for $C\in\Mat(2n_2\times2n_1,\C)\Rightarrow C=\sigma\otimes C_0$ where $C_0\in\Mat(n_2\times n_1,\C)$.
\end{enumerate}
\end{lemma}
\begin{proof}
We just prove (1) since (2) follows similarly and (3) is a consequence of (1) and (2). Let $A_0=\mbox{diag}(a_1,...,a_{n_1})$ and denote the first row of $C$ by
$(c_1,...,c_{2n_1})$. Multiplication of the first row of $C$ with the first column of $A$ yields $(c_1+ic_{n_1+1})a_1=0$. Since $A$ is positive definite we get
$c_{n_1+1}=ic_1$. Analogously we obtain $c_{n_1+j}=ic_{j}$ for $1\leq j\leq n_1$. The claim is established by proceeding analogously for the remaining rows of $C$.
\end{proof}

\subsection{Isospectral families of focal shape operators for the case $m=1$}
\label{ch3}

In this subsection we prove our main theorem and assume $(g,m)=(6,1)$ throughout.

\begin{theorem}
\label{zus}
Let $L(t)=\cos(t)L_0+\sin(t)L_1\in\mbox{Sym}(5,\R)$, $t\in\R$, be isospectral where $L_0=\mbox{diag}(\sqrt{3},\tfrac{1}{\sqrt{3}},0,-\tfrac{1}{\sqrt{3}},-\sqrt{3})$. Up to conjugation by an element $A\in\OO(5)$ with $AL_0A^{-1}=L_0$, the matrix $L_1$
is given by one of the following matrices
\begin{align*}
&L_1=\tfrac{1}{3\sqrt{3}}\left(\begin{smallmatrix}
0&5&0&2&0\\
5&0&4&0&2\\
0&4&0&4&0\\
2&0&4&0&5\\
0&2&0&5&0
\end{smallmatrix} \right),\,L_1=\tfrac{1}{\sqrt{6}}\left(\begin{smallmatrix}
0&0&0&0&3\sqrt{2}\\
0&0&1&0&0\\
0&1&0&1&0\\
0&0&1&0&0\\
3\sqrt{2}&0&0&0&0
\end{smallmatrix} \right),\,L_1=\tfrac{1}{\sqrt{6}}\left(\begin{smallmatrix}
0&0&3&0&0\\
0&0&0&\sqrt{2}&0\\
3&0&0&0&3\\
0&\sqrt{2}&0&0&0\\
0&0&3&0&0
\end{smallmatrix} \right),\\
&L_1=\tfrac{1}{\sqrt{3}}\left(\begin{smallmatrix}
0&1&0&2&0\\
1&0&0&0&2\\
0&0&0&0&0\\
2&0&0&0&1\\
0&2&0&1&0
\end{smallmatrix} \right),\,L_1=\tfrac{1}{\sqrt{3}}\left(\begin{smallmatrix}
0&\sqrt{3}&0&0&0\\
\sqrt{3}&0&0&2&0\\
0&0&0&0&0\\
0&2&0&0&\sqrt{3}\\
0&0&0&\sqrt{3}&0
\end{smallmatrix} \right),\,L_1=\tfrac{1}{\sqrt{3}}\left(\begin{smallmatrix}
0&0&0&0&3\\
0&0&0&1&0\\
0&0&0&0&0\\
0&1&0&0&0\\
3&0&0&0&0
\end{smallmatrix} \right).
\end{align*}
For these cases $\dim(E)$ is given by $3,2,2,1,1$ and $1$, respectively.
\end{theorem}

The proof of this theorem consists of the Lemmas of this subsection.

\begin{remark}
For the case $(g,m)=(6,2)$, in which the matrices are 10 by 10, there does not yet exist a classification of the isospectral families of focal shape operators.
\end{remark}

\begin{lemma}
\label{structuree4}
Up to conjugation by an element of $\OO_{5}(\R)$ the matrix $P_4=L_+^4$ is of the form
$P_4=\bigl(\begin{smallmatrix}
A&0\\ 0&0
\end{smallmatrix} \bigr)$
with $A=\sigma\otimes A_0,$
where $A_0\in\R$.
\end{lemma}
\begin{proof}
Since $P_4^2=L_+^8=0$, Lemma\,\ref{stl} implies $P_4=\bigl(\begin{smallmatrix}
A&0\\ 0&0
\end{smallmatrix} \bigr)$
with $A=\sigma\otimes A_0,$
where $A_0\in\Sym(d_1,\R)$ is a diagonal and positive definite matrix or the null matrix with $2d_1\leq 5$.
In what follows we assume $d_1=2$. Thus by $L_+P_4=P_4L_+=0$ and Lemma\,\ref{multi} we have
$L_+=\left(\begin{smallmatrix}
\sigma\otimes U_0&e_+\otimes V_0\\
e_+^{tr}\otimes {V_0}^{tr}&W_0
\end{smallmatrix} \right)$, where  $U_0\in\Mat(2,\C)$,  $V_0\in\Mat(2\times 1,\C)$  and $W_0\in\C$.
Thus we get $P_4=L_+^4=\left(\begin{smallmatrix}
W_0^2(\sigma\otimes V_0V_0^{tr})&W_0^3(e_+\otimes V_0)\\
W_0^3(e_+^{tr}\otimes {V_0}^{tr})&W_0^4
\end{smallmatrix} \right)$ and therefore $W_0=0$. However this implies $P_4=0$, which contradicts our assumption. Thus $d_1\in\lbrace 0,1\rbrace$.
\end{proof}

By Lemma\,\ref{cmp} we have $L_+^5=0$. Below we consider successively the four cases $L_+^{j+1}=0$, $L_+^j\neq 0$, $j\in\lbrace1,2,3,4\rbrace$, 
and determine the possible $L_1$ for each case.

\begin{lemma}
\label{prol2}
Let $P_4=\bigl(\begin{smallmatrix}
A&0\\ 0&0
\end{smallmatrix} \bigr)$
with $A=\sigma\otimes A_0,$
where $A_0\in\R-\lbrace 0\rbrace$.
Up to conjugation by an element of $\OO_{5}(\R)$ the matrix ${L_+^2}$ is of the form
$$L_+^2=\left(\begin{smallmatrix}
\sigma\otimes B_0&\sigma\otimes C_1&e_+\otimes C_2\\
\sigma\otimes C_1^{tr}&\sigma\otimes D_0&0\\
e_+^{tr}\otimes C_2^{tr}&0&0
\end{smallmatrix} \right)$$
where $B_0\in\C$, $D_0\in\Sym(d_2,\R)$ is a diagonal, positive definite matrix or the null matrix with $2d_2\leq 3$,
$C_1\in\mbox{Mat}(d_2,\C)$, $C_2\in\Mat(1\times(3-2d_2),\C)$ and $A_0=C_2\,C_2^{tr}$.
\end{lemma}
\begin{proof}
Introduce the notation $L_+^2=\bigl(\begin{smallmatrix}
B&C\\
C^{tr}&D
\end{smallmatrix} \bigr)$
where $B\in\Sym(2d_1,\C)$.
Hence $L_+^2\,P_4=P_4\,L_+^2=0$ imply $B\,A=0,\,A\,B=0$ and $A\,C=0.$
By Lemma \ref{multi} we get $B=\sigma\otimes B_0$
with $B_0\in\C$ and $C=e_+\otimes C_0.$
Calculating $(L_+^2)^2$ and using Lemma\,\ref{structuree4} we get $A=CC^{tr}$, $DC^{tr}=0$ and $D^2=0$.
In particular, $A_0=C_0\,C_0^{tr}.$
Since $D^2=0$, Lemma\,\ref{stl} implies $D=\bigl(\begin{smallmatrix}
\sigma\otimes D_0&\\
&0
\end{smallmatrix} \bigr)$, where $D_0\in\Sym(d_2,\R)$ is a diagonal and positive definite matrix. 
From $D\in\Sym(3,\C)$ we get $2d_2\leq 3.$
Lemma \ref{multi} yields $C^{tr}=\left(\begin{smallmatrix}
\sigma\otimes C_1^{tr}\\
e_+^{tr}\otimes C_2^{tr}
\end{smallmatrix} \right)$, where $C_1\in\Mat(1\times d_2,\C)$ and $C_2\in\Mat(1\times(3-2d_2),\C).$
Finally, $\sigma\otimes A_0=C\,C^{tr}=\sigma\otimes (C_2\, C_2^{tr})$
implies $A_0=C_2\,C_2^{tr}.$ 
\end{proof}

\begin{lemma}
\label{lplus}
Assume $P_4\neq 0$.
Up to conjugation by an element of $\OO_{5}(\R)$ the matrix $L_+$ is of the form
 $$L_+=\bigl(\begin{smallmatrix}
0&e_+\otimes F_0\\
e_+^{tr}\otimes F_0^{tr}&G
\end{smallmatrix} \bigr)$$
where $G=\bigl(\begin{smallmatrix}
\sigma\otimes G_1&e_+\otimes G_3\\
e_+^{tr}\otimes G_3&0
\end{smallmatrix} \bigr)\in\Sym(3,\C)$ with $G_1\in\C$.
Furthermore, 
\begin{align*}
B_0=F_0\,F_0^{tr},\hspace{1cm}C_0=F_0\, G,\hspace{1cm}D=G^2,\hspace{1cm}A_0=F_0\, D\,F_0^{tr}
\end{align*}
and $d_2=1$. Finally, $D_0=G_3^2$.
\end{lemma}

\begin{proof}
Using Lemma\,\ref{structuree4} and $L_+P_4=P_4L_+=L_+^5=0$ we deduce $$L_+=\bigl(\begin{smallmatrix}
\sigma\otimes E_0&e_+\otimes F_0\\
e_+^{tr}\otimes F_0^{tr}&G
\end{smallmatrix} \bigr)$$
for an $E_0\in\C$.
By Lemma \ref{ortlem} we get $P_4L_-P_4=0$
which is equivalent to $E_0=0$.
Calculating $L_+^2$ and using Lemma\,\ref{prol2}, we obtain the first three of the claimed identities. Plugging $C_0=F_0G$ into $A_0=C_0\,C_0^{tr}$ and using $D=G^2$ we obtain the fourth equation, which implies that $D_0$ cannot vanish, i.e., $d_2=1$. Decomposing $G$
corresponding to $D$ and evaluating $D=G^2$ yields that $G$ is of the stated form.
\end{proof}

\begin{lemma}
\label{falle4nn}
If $\rk P_4=1$ then there exists an $A\in\OO(5)$ such that $AL(t)A^{-1}=\cos(t)L_0+\sin(t)L_1$ with
$L_0=\mbox{diag}(\sqrt{3},\tfrac{1}{\sqrt{3}},0,-\tfrac{1}{\sqrt{3}},-\sqrt{3})$ and $$L_1=\tfrac{1}{3\sqrt{3}}\left(\begin{smallmatrix}
0&5&0&2&0\\
5&0&4&0&2\\
0&4&0&4&0\\
2&0&4&0&5\\
0&2&0&5&0
\end{smallmatrix} \right).$$
In particular, $\dim{E}=3$.
\end{lemma}
\begin{proof}
Introduce the notation $F_0=(e_+^{tr}\otimes F_1+e_-^{tr}\otimes F_2,F_3)$ with $F_1,F_2,F_3\in\C$. Then $A_0=F_0\, D\,F_0^{tr}$ is equivalent to $A_0=F_2^2\,D_0$ which implies $F_2\in\R^*$. From the $(4,5)$-component of $L_+\,P_2+L_-\,P_4=0$ and $F_2,G_3\in\R^*$ we have $F_1=0$.
Thus the $(5,1)$-component of $L_+\,P_2+L_-\,P_4=0$ yields $G_1=-2\tfrac{F_2F_3}{G_3}.$ Therefore the $(5,5)$-component of $L_+\,P_0+L_-\,P_2=0$ implies $F_3=0$. Hence $L_+\,P_2+L_-\,P_4=0$ is equivalent to
$3G_3^2+3F_2^2-5=0$ and $L_+P_0+L_-P_2=0$ reduces to $3-10F_2^2+3F_2^4=0.$ Consequently, $F_2=\pm\tfrac{1}{\sqrt{3}}$ or $F_2=\pm\sqrt{3}.$
If $F_2=\pm\sqrt{3}$ we have $\mbox{Im}G_3\neq 0$ which contradicts $G_3\in\R^*$.
Thus $F_2=\pm\tfrac{1}{\sqrt{3}}$. Consequently,
$(F_2,G_3)\in\left\{(\tfrac{1}{\sqrt{3}},\tfrac{2}{\sqrt{3}}),(\tfrac{1}{\sqrt{3}},-\tfrac{2}{\sqrt{3}}),(-\tfrac{1}{\sqrt{3}},\tfrac{2}{\sqrt{3}}),(-\tfrac{1}{\sqrt{3}},-\tfrac{2}{\sqrt{3}})\right\}.$
We determine $L_0$ and $L_1$ for each of these cases and
perform a change of the basis such that the basis consists of unit eigenvectors of $L_0$.
If $(F_2,G_3)=(\frac{1}{\sqrt{3}},\frac{2}{\sqrt{3}})$ or $(F_2,G_3)=(-\tfrac{1}{\sqrt{3}},-\tfrac{2}{\sqrt{3}})$ we obtain the above $L_1$ with the $+$-sign.
For the remaining two cases the sign of $L_1$ changes,
which corresponds to a change of orientation of $(M,g_0)$. Conjugating $\cos(t)L_0+\sin(t)L_1$ by $\mbox{diag}(\unaryminus1,1,\unaryminus1,1,\unaryminus1)$ the claim follows.
\end{proof}

\begin{lemma}
For $P_4=0$ and $L_+^3\neq 0$ there exists an $A\in\OO(5)$ such that $AL(t)A^{-1}=\cos(t)L_0+\sin(t)L_1$ with
$L_0=\mbox{diag}(\sqrt{3},\tfrac{1}{\sqrt{3}},0,-\tfrac{1}{\sqrt{3}},-\sqrt{3})$ and
$$L_1=\tfrac{1}{\sqrt{3}}\left(\begin{smallmatrix}
0&\sqrt{3}&0&0&0\\
\sqrt{3}&0&0&2&0\\
0&0&0&0&0\\
0&2&0&0&\sqrt{3}\\
0&0&0&\sqrt{3}&0
\end{smallmatrix} \right).$$
In particular $\dim E=1$.
\end{lemma}

\begin{proof}
By $(L_+^3)^2=0$ and Lemma\,\ref{stl} we get
$L_+^3=\bigl(\begin{smallmatrix}
\sigma\otimes S_0&0\\ 0&0
\end{smallmatrix} \bigr)$ where $S_0\in\mbox{Mat}(d_3,\R)$ is a positive definite, diagonal matrix. Therefore $d_3\in
\left\{1,2\right\}.$
Introduce the notation $L_+=\bigl(\begin{smallmatrix}
T&U\\
U^{tr}&V
\end{smallmatrix} \bigr)$ 
where $T\in\mbox{Mat}(2d_3,\C)$.
From $L_+\,L_+^3=0$ and $L_+^3\,L_+=0$ we get $T=\sigma\otimes T_0$ with $T_0\in\mbox{Mat}(d_3,\C)$ and $U=e_+\otimes U_0$.
Furthermore, the identity $L_+\,P_2=0$ implies $L_+^3\,P_2=0$, which is equivalent to $T_0=0.$
Hence $L_+=\bigl(\begin{smallmatrix}
0&e_+\otimes U_0\\
e_+^{tr}\otimes U_0^{tr}&V
\end{smallmatrix} \bigr).$
Calculating $L_+^3$ and comparing with $L_+^3=\bigl(\begin{smallmatrix}
\sigma\otimes S_0&0\\ 0&0
\end{smallmatrix} \bigr)$ yields $V^3=0$ and $(e_+\otimes U_0)V^2=0$.
If $d_3=2$ we have $V\in\C$ and thus $V=0$, which yields $S_0=0$, contradicting our assumption.
Thus $d_3=1$.
From $(V^2)^2=0$ and Lemma\,\ref{stl} we have $V^2=\bigl(\begin{smallmatrix}
\sigma\otimes W&\\
&0
\end{smallmatrix} \bigr)$ where $W\geq 0$.

\smallskip

First let $W=0$ and thus $V^2=0.$ Using Lemma\,\ref{stl} we get $V=\bigl(\begin{smallmatrix}
\sigma\otimes V_0&0\\ 0&0
\end{smallmatrix} \bigr)$ for $V_0\in\R$. Since $V_0=0$ would imply $S_0=0$, we have $V_0>0$.
Introduce $u_1,u_2,u_3\in\C$ by $U_0=(u_1,u_2,u_3).$ Calculating $L_+^3$ yields $S_0=\frac{1}{4}(u_1+iu_2)^2V_0.$
Since $S_0>0$ we get $u_1+iu_2\in\R^{*}.$ 
Therefore the $(4,2)$ equation of $L_+\,P_2=0$ is equivalent to $\overline{u_1}+i\,\overline{u_2}=0$. 
Combining this equation with $u_1+iu_2\in\R^{*}$ yields $u_1\in\R$ and $u_2\in i\R.$ 
Hence the $(5,5)$ equation of $L_+\,P_0+L_-\,P_2=0$ is equivalent to
 $u_2^2\overline{u_3}^2=0.$ Since $u_2=0$ would imply $L_+^3=0$ we get $u_3=0.$
Thus $L_+\,P_2=0$ is equivalent to $V_0^2=(10+12u_2^2)/3.$ Plugging this into $L_+\,P_0+L_-\,P_2=0$ yields
$u_2=\pm\tfrac{i}{\sqrt{2}}$ and thus $V_0=\tfrac{2}{\sqrt{3}}.$
For both possible cases we obtain $-L_1.$ Conjugating by $\mbox{diag}(\unaryminus1,1,1,\unaryminus1,1)$ yields the claim.

\smallskip

If $W>0$ the equation $(e_+\otimes U_0)V^2=0$
implies $u_1=-i\,u_2$, where the $u_i$ are as above. Since $u_1=-i\,u_2$ yields $L_+^3=0$, the case $W>0$ cannot occur.
\end{proof}

\begin{lemma}
\label{ph2}
For $L_+^3=0$ and $L_+^2\neq 0$ there exists an $A\in\OO(5)$ such that $AL(t)A^{-1}=\cos(t)L_0+\sin(t)L_1$ with
$L_0=\mbox{diag}(\sqrt{3},\tfrac{1}{\sqrt{3}},0,-\tfrac{1}{\sqrt{3}},-\sqrt{3})$ and
$$L_1=\tfrac{1}{\sqrt{6}}\left(\begin{smallmatrix}
0&0&0&0&3\sqrt{2}\\
0&0&1&0&0\\
0&1&0&1&0\\
0&0&1&0&0\\
3\sqrt{2}&0&0&0&0
\end{smallmatrix} \right), L_1=\tfrac{1}{\sqrt{6}}\left(\begin{smallmatrix}
0&0&3&0&0\\
0&0&0&\sqrt{2}&0\\
3&0&0&0&3\\
0&\sqrt{2}&0&0&0\\
0&0&3&0&0
\end{smallmatrix} \right)\hspace{0.2cm}\mbox{or}\hspace{0.2cm} L_1=\tfrac{1}{\sqrt{3}}\left(\begin{smallmatrix}
0&1&0&2&0\\
1&0&0&0&2\\
0&0&0&0&0\\
2&0&0&0&1\\
0&2&0&1&0
\end{smallmatrix} \right).$$ In particular $\dim E=2$ for the first two cases and $\dim E=1$ for the last one.
\end{lemma}
\begin{proof}
Using Lemma\,\ref{stl} identity $(L_+^2)^2=0$ implies $L_+^2=\bigl(\begin{smallmatrix}
\sigma\otimes S_0&0\\
0&0
\end{smallmatrix} \bigr)$,
where $S_0$ is a positive definite, diagonal matrix.
Introduce the notation $L_+=\bigl(\begin{smallmatrix}
T&U\\
U^{tr}&W
\end{smallmatrix} \bigr)$.
From $L_+^2\,L_+=0=L_+\,L_+^2$ we have
$T=\sigma\otimes T_0$ and $U=e_+\otimes U_0$.
Since $L_+\,P_2=0$ is equivalent to $T_0=0$ we get
$L_+=\bigl(\begin{smallmatrix}
0&e_+\otimes U_0\\
e_+^{tr}\otimes U_0^{tr}&W
\end{smallmatrix} \bigr)$.
Calculating $L_+^2$ and using Lemma\,\ref{prol2} yields $W^2=0$, $(e_+\otimes U_0)W=0$ and $S_0=U_0\,U_0^{tr}$.

\smallskip

First we suppose $S_0\in\Mat(2,\R)$ which implies $W\in\C$. Hence $W^2=0$ implies $W=0$ and therefore $\rk L(t)\leq 2$ for all $t\in\R$, which is a contradiction.

\smallskip

Next let $S_0\in\R$ and introduce the notation $U_0=(u_1,u_2,u_3)$ with $u_i\in\C$.
By Lemma\,\ref{stl} and $W^2=0$ we have $W=\bigl(\begin{smallmatrix}
\sigma\otimes W_0&\\
&0
\end{smallmatrix} \bigr)$, where $W_0\geq 0$.\\ First let $W_0=0$. There exists an $s_1\in\R$ such that conjugating $L_{\pm}$ by $T_1=\left(\begin{smallmatrix}
\eins_2&&\\
&D(s_1)&\\
&&1
\end{smallmatrix} \right)$, where $D(t)=\bigl(\begin{smallmatrix}
\cos(t)&\sin(t)\\
-\sin(t)&\cos(t)
\end{smallmatrix} \bigr)$, transforms $U_0$ into the form $U_0=(u_1,u_2,u_3)$ with $u_1\in\R$ and $u_2,u_3\in\C$.
Similarly, there exists an $s_2\in\R$ such that conjugating $T_1L_{\pm}T_1^{-1}$ by $T_2=\bigl(\begin{smallmatrix}
\eins_3&\\
&D(s_2)
\end{smallmatrix} \bigr)$ transforms $U_0$ into the form $U_0=(u_1,u_2,u_3)$ with $u_1,u_2\in\R$ and $u_3\in\C$.
Finally, there exists an $s_3\in\R$ such that conjugating $T_2T_1L_{\pm}T_1^{-1}T_2^{-1}$ by $T_3=\left(\begin{smallmatrix}
\eins_2&&\\
&D(s_3)&\\
&&1
\end{smallmatrix} \right)$ transforms $U_0$ into the form $U_0=(0,u_2,u_3)$ with $u_2\in\R$ and $u_3\in\C$.
Since $S_0=U_0\,U_0^{tr}$ and $S_0\in\R$ we thus get $u_3\in\R$ or $u_3\in i\R$.
One proves easily that the eigenvalues of $L(t)$ are given by $0$ and $\pm\sqrt{\sum_{i=1}^3(u_i\overline{u_i}\pm u_i^2)}$
and thus in the former case at least three eigenvalues vanish. Consequently, $u_3\in i\R$. Hence $L_+P_0+L_-P_2=0$ implies
$(u_2,u_3)=(\pm \tfrac{\sqrt{3}}{2},\pm i\tfrac{1}{2\sqrt{3}})$.
It is straightforward to verify that for each of these cases there exists an $A\in\OO(5)$ such that $AL(t)A^{-1}=\cos(t)L_0+\sin(t)L_1$ with
$L_0=\mbox{diag}(\sqrt{3},\tfrac{1}{\sqrt{3}},0,-\tfrac{1}{\sqrt{3}},-\sqrt{3})$ and $L_1$ is given by $$L_1=\tfrac{1}{\sqrt{3}}\left(\begin{smallmatrix}
0&1&0&2&0\\
1&0&0&0&2\\
0&0&0&0&0\\
2&0&0&0&1\\
0&2&0&1&0
\end{smallmatrix} \right).$$

Below we assume $W_0>0$. 
The identity $(e_+\otimes U_0)W=0$ yields $u_1=-i\,u_2$. Thus the $(5,5)$ equation of $L_+\,P_0+L_-\,P_2=0$ is given by
$W_0\,\overline{u_2}^2\,u_3^2=0.$ Since $u_3=0$ would imply $L_+^2=0$ we have $u_2=0$.
Consequently, $S_0=U_0U_0^{tr}$ is equivalent to $S_0=u_3^2$ and thus we have
$u_3\in\R^*$. Hence $L_+\,P_0+L_-\,P_2=0$ yields $W_0\in\left\{1/\sqrt{3},\sqrt{3}\right\}$ and $u_3\in\left\{\pm 1/\sqrt{6},\pm\sqrt{3/2}\right\}$.
From $\spec{L(t)}=\left\{0,\pm\sqrt{2}u_3,\pm W_0\right\}$ we thus get (i) $(W_0,u_3)=(1/\sqrt{3},\sqrt{3/2})$, (ii) $(W_0,u_3)=(1/\sqrt{3},-\sqrt{3/2})$, (iii) $(W_0,u_3)=(\sqrt{3},1/\sqrt{6})$ or (iv) $(W_0,u_3)=(\sqrt{3},-1/\sqrt{6})$.
We determine $L_0$ and $L_1$ for each of these cases and
perform a change of the basis such that the basis consists of unit eigenvectors of $L_0$, more precisely $L_0=\mbox{diag}(\sqrt{3},\tfrac{1}{\sqrt{3}},0,-\tfrac{1}{\sqrt{3}},-\sqrt{3})$. For the cases (i)-(iv) we get
\begin{align*}
&L_1=\tfrac{1}{\sqrt{6}}\left(\begin{smallmatrix}
0&0&3&0&0\\
0&0&0&-\sqrt{2}&0\\
3&0&0&0&3\\
0&-\sqrt{2}&0&0&0\\
0&0&3&0&0
\end{smallmatrix} \right), \hspace{1cm}L_1=-\tfrac{1}{\sqrt{6}}\left(\begin{smallmatrix}
0&0&3&0&0\\
0&0&0&\sqrt{2}&0\\
3&0&0&0&3\\
0&\sqrt{2}&0&0&0\\
0&0&3&0&0
\end{smallmatrix} \right),\\&L_1=-\tfrac{1}{\sqrt{6}}\left(\begin{smallmatrix}
0&0&0&0&3\sqrt{2}\\
0&0&1&0&0\\
0&1&0&1&0\\
0&0&1&0&0\\
3\sqrt{2}&0&0&0&0
\end{smallmatrix} \right),\hspace{1cm} L_1=\tfrac{1}{\sqrt{6}}\left(\begin{smallmatrix}
0&0&0&0&-3\sqrt{2}\\
0&0&1&0&0\\
0&1&0&1&0\\
0&0&1&0&0\\
-3\sqrt{2}&0&0&0&0
\end{smallmatrix}\right),
\end{align*}
 respectively.
Conjugating the matrices of the first row by the matrices $\mbox{diag}(1,\unaryminus1,1,1,1)$ and $\mbox{diag}(\unaryminus1,\unaryminus1,1,1,\unaryminus1)$, respectively, 
and the matrices of the second row by the matrices $\mbox{diag}(\unaryminus1,\unaryminus1,1,\unaryminus1,1)$ and $\mbox{diag}(\unaryminus1,1,1,1,1)$, respectively,  the claim follows.
\end{proof}

\begin{lemma}
If $L_+^2=0$ there exists an $A\in\OO(5)$ such that $AL(t)A^{-1}=\cos(t)L_0+\sin(t)L_1$ with
$L_0=\mbox{diag}(\sqrt{3},\tfrac{1}{\sqrt{3}},0,-\tfrac{1}{\sqrt{3}},-\sqrt{3})$ and
$$L_1=\tfrac{1}{\sqrt{3}}\left(\begin{smallmatrix}
0&0&0&0&3\\
0&0&0&1&0\\
0&0&0&0&0\\
0&1&0&0&0\\
3&0&0&0&0
\end{smallmatrix} \right).$$ In particular $\dim E=1$.
\end{lemma}
\begin{proof}
From $L_+^2=0$ we have 
$L_+=\left(\begin{smallmatrix}\sigma\otimes S_0&\\
&0\end{smallmatrix} \right)$, where $S_0\in\mbox{Mat}(d,\R)$ is a diagonal and positive definite matrix.
Furthermore, $P_{\pm 4}=0=P_{\pm 2}$ and thus $P(t)=P_0$, which implies $P_0^2=P_0$.
Introduce the notation $P_0=\left(\begin{smallmatrix} T&U\\
U^{tr}&V\end{smallmatrix} \right)$, where $T\in\mbox{Mat}(2d,\C)$. By the very definition of $P_0$ we get $U=0$ and $V=\eins$.
The equations $L_+P_0=0=P_0L_+$ imply $T=\sigma \otimes T_0$ for $T_0\in\mbox{Mat}(d,\R)$. Therefore $P_0^2=P_0$ yields $T_0=0$.
Consequently, $S_0^4-\tfrac{10}{3}S_0^2+\eins_d=0$.
If $d=1$ we get $\rk P_0=3$ which contradicts our assumption. Hence $d=2$. Thus we obtain $S_0=\mbox{diag}(\sqrt{3},\tfrac{1}{\sqrt{3}})$ or $S_0=\mbox{diag}(\tfrac{1}{\sqrt{3}},\sqrt{3})$. In the former case the claim follows by conjugation by $\mbox{diag}(\unaryminus1,\unaryminus1,1,1,1)$,
the remaining cases is treated similarly.
\end{proof}

By combining the previous results we finally obtain Theorem\,\ref{zus}.

\section{Classification of isoparametric hypersurfaces with $(g,m)=(6,1)$}
\label{sec4}
After giving a very short exposition to isoparametric hypersurfaces in spheres in Subsection\,\ref{21}, we explain in Subsection\,\ref{22} the significance of Theorem\,\ref{zus}
 in the context of isoparametric hypersurfaces in spheres. Finally, in Subsection\,\ref{23} we  show that all isoparametric hypersurfaces in $\Sph^7$ with $g=6$ are homogeneous and thereby reprove a result of Dorfmeister and Neher \cite{dn}.

\subsection{Isoparametric hypersurfaces in spheres}
\label{21}
Hypersurfaces in spheres with constant principal curvatures are called isoparametric.
M\"unzner \cite{munzner, munzner2} showed that the number of distinct principal
curvatures $g$ can be only $1$, $2$, $3$, $4$, or $6$, and gave restrictions for the
multiplicities as well. The possible multiplicities of the curvature distributions were classified in \cite{munzner2,abresch,stolz},
and coincide with the multiplicities in the known examples. So far the cases $g = 4$ and $g=6$ are not yet completely classified.
See e.g. the paper \cite{tb1} of Thorbergsson for a survey of isoparametric hypersurfaces in spheres.

\smallskip

For the case $g=6$ all multiplicities coincide and are given either by $m=1$ or $m=2$. 
Furthermore, exactly two examples are known for this case, both of which are homogeneous. They are given as orbits of the
isotropy representation of $\Gtwo/\SO(4)$ or as orbits in the unit sphere $\Sph^{13}$ of the Lie
algebra $\frak{g}_2$ of the adjoint representation of the Lie group $\Gtwo$ and have multiplicities $m=1$ and $m=2$, respectively.
Dorfmeister and Neher \cite{dn} conjectured that all isoparametric hypersurfaces with $g=6$ are homogeneous and proved this in the affirmative 
for the case $m=1$. Since homogeneous isoparametric hypersurfaces in spheres were classified by Takagi and Takahashi \cite{tt}, this provides a classification of 
isoparametric hypersurfaces with $(g,m)=(6,1)$.  The case $m=2$ is not classified yet.

\subsection{Link of isoparametric hypersurfaces to Theorem\,\ref{zus}}
\label{22}
Throughout this paper $M$  denotes a connected, smooth manifold of dimension $n$.
An embedding $F_{0}:{M}\hookrightarrow \Sph^{n+1}$ together 
with a distinguished unit normal vector field $\nu_0\in\Gamma(\nu M)$ is called an {\em isoparametric hypersurface in $\Sph^{n+1}$}
if and only if its principal curvatures are constant. We denote by $A_0$ the shape operator of $F_0$ with respect to $\nu_0$ and by $\lambda^{0}_{j}$, $j\in\left\{1,...,g\right\}$,
the principal curvatures. We further assume without loss of generality $\lambda^{0}_1>...>\lambda^0_g$ and
define $\theta_j\in\left(-\frac{\pi}{2},\frac{\pi}{2}\right)$ such
that $\lambda^0_j=\cot(\theta_j)$.
It is well-known that the $j$-th curvature distribution $D_j$, which is given by $D_j(p)=\mbox{Eig}({A_0}_{\lvert p},\lambda_j^0)$ for $p\in M$, is integrable and its leaves $\frak{L}_j$ are small spheres in $\Sph^{n+1}$.

\smallskip

We consider the \textit{parallel surface} $F_{s}: M\hookrightarrow\Sph^{n+1}$  defined via
\begin{align*}
p\mapsto F_s(p):=\exp_{F_0(p)}(s\nu_{0\vert p})=\cos(s)F_0(p)+\sin(s)\nu_{0\vert p},
\end{align*}
endowed with the orientation $\nu_{s}(p)=-\sin(s)F_{0}(p)+\cos(s){\nu_{0}}_{\vert_p}$.
If $s\neq\theta_j$, the parallel surface $F_s(M)$ is again an isoparametric hypersurface with principal curvatures $\lambda^s_j=\cot(\theta_j-s)$.
For $s_0=\theta_j+\ell\pi$, $\ell\in\Z_2$, the map $F_{s_0}$ focalizes $\frak{L}_j(p)$ to one point in the $(n-m_j)$-dimensional focal submanifold $M_{j,\ell}:=F_{\theta_j+\ell\pi}(M)$. 

\smallskip

Let $\ell\in\Z_2$ be given. M\"unzner \cite{munzner} proved that the spectrum of the shape operator 
$\mathcal{A}_{\nu_{\lvert \overline{p}}}$ of $M_{j,\ell}$ is independent of $\nu\in\nu M_{j,\ell}$ and $\overline{p}\in M_{j,\ell}$
and is given by  
\begin{align*}
\mbox{spec}(\mathcal{A}_{\nu_{\lvert \overline{p}}})=\left\{\cot\big((i-j)\pi/g \big)\, \vert\, i\in \left\{1,...,g\right\},\, i\neq j\right\}.
\end{align*}
Thus for each $\overline{p}\in M_{j,\ell}$ and each pair of orthonormal vectors $v_0,v_1\in\nu_{\overline{p}} M_{j,\ell}$
the family of shape operators $L(t)=\mathcal{A}_{\cos(t)v_0+\sin(t)v_1}=\cos(t)\mathcal{A}_{v_0}+\sin(t)\mathcal{A}_{v_1}, t\in\R$, is isospectral. 
We introduce the short-hand notation $L=L(t)$, $L_0=\mathcal{A}_{v_0}$, $L_1=\mathcal{A}_{v_1}$. 
Consequently, if we restrict ourselves to the case $(g,m)=(6,1)$ we get an isospectral family $L(t)$ with spectrum $\lbrace -\sqrt{3},-\tfrac{1}{\sqrt{3}},0,\tfrac{1}{\sqrt{3}},\sqrt{3}\rbrace$,
where all eigenvalues have multiplicity $1$ - these families are classified in Theorem\,\ref{zus}.

\subsection{Proof of homogeneity}
\label{23}
Takagi and Takahashi \cite{tt} classified homogeneous isoparametric hypersurfaces in spheres and
in particular showed that for $(g,m)=(6,1)$ there is only one example, namely the isoparametric hypersurface is given by orbits of the isotropy representation of $\Gtwo/\SO(4)$.
Before proving the main result of this section we consider the homogeneous example more detailed.

\smallskip

In \cite{siff} the symmetric, trilinear form $\alpha$ was introduced by $$\alpha(\,\cdot\,,\,\cdot\,,\,\cdot\,)=g_0((\nabla^0_{\,\cdot\,}A_0)\,\cdot\,,\,\cdot\,),$$ where
$g_0=F_0^{*}\left\langle \cdot,\cdot\right\rangle_{\Sph^{n+1}}$ and $\nabla^0$ is
the associated Levi-Civita connection.
Furthermore,
it was shown, using the computations in \cite{mi33},
that for the homogeneous example with $(g,m) = (6,1)$ the components $\alpha_{i,\,j,\,k}:=\alpha(e_i,e_j,e_k)$ are given by
\begin{align}
\label{ex1}
\alpha_{1,\,2,\,3}=\alpha_{3,\,4,\,5}=\alpha_{1,\,5,\,6}=\sqrt{\tfrac{3}{2}},\,\alpha_{2,\,4,\,6}=-\sqrt{\tfrac{3}{2}},\,\alpha_{1,\,3,\,5}=-2\sqrt{\tfrac{3}{2}}
\end{align}
and
all other $\alpha_{i,\,j,\,k}$ with $i\leq j\leq k$ vanish, or by 
\begin{align}
\label{ex2}
\alpha_{4,\,5,\,6}=\alpha_{2,\,3,\,4}=\alpha_{1,\,2,\,6}=\sqrt{\tfrac{3}{2}},\,\alpha_{1,\,3,\,5}=-\sqrt{\tfrac{3}{2}},\,\alpha_{2,\,4,\,6}=-2\sqrt{\tfrac{3}{2}},
\end{align}
and all other $\alpha_{i,\,j,\,k}$ with $i\leq j\leq k$ vanish. Note that (\ref{ex2}) is obtained from (\ref{ex1}) by flipping the orientation of the isoparametric hypersurface.

\smallskip

The strategy for proving that all isoparametric hypersurfaces in $\Sph^7$ with $g=6$ are homogeneous is as follows: the main step consists in the proof of the 
fact that for these isoparametric hypersurfaces all $\alpha_{i,\,j,\,k}^2$ coincide with those of the homogeneous example, i.e., either with (\ref{ex1}) or (\ref{ex2}). Then the desired result follows by the following proposition of Abresch.

\begin{proposition}[see Proposition\,12.5 in \cite{abresch}]
\label{ab}
Isoparametric hypersurfaces ${M}\subset\Sph^{7}$ with $g=6$ are homogeneous if and only if all the functions $\alpha_{i,\,j,\,k}^2$ are constant on ${M}\subset\Sph^{7}$. 
\end{proposition}

\begin{theorem}
Isoparametric hypersurfaces in $\Sph^{7}$ with $(g,m)=(6,1)$ are homogeneous.
\end{theorem}

\begin{proof}
Let an isoparametric hypersurface $M$ and $p\in M$ be given.
Recall that for $s_0=\theta_i+\ell\pi$, $\ell\in\Z_2$, the map $F_{s_0}$ focalizes $\frak{L}_i(p)$ to one point $p_{i,\ell}\in M_{i,\ell}$. 
Let the isospectral family of focal shape operators at this point be denoted by $L^{i,\ell}(t)=\cos(t)L_0+\sin(t)L_1^{i,\ell}$ und assume 
$L_0=\mbox{diag}(\sqrt{3},\frac{1}{\sqrt{3}},0,-\frac{1}{\sqrt{3}},-\sqrt{3})$. Then it is easy to prove that $L_1^{i,\ell}$ is given by
\begin{align*}
(-1)^{\ell}\left(\begin{array}{ccccc}
0&\sqrt{\frac{2}{3}}\,\alpha_{i,\,{i+1},\,{i+2}}&\frac{1}{\sqrt{2}}\,\alpha_{i,\,i+1,\,i+3}&\sqrt{\frac{2}{3}}\,\alpha_{i,\,i+1,\,i+4}&\sqrt{2}\,\alpha_{i,\,i+1,\,i+5}\\
\sqrt{\frac{2}{3}}\,\alpha_{i,\,i+1,\,i+2}&0&\frac{1}{\sqrt{6}}\,\alpha_{i,\,i+2,\,i+3}&\frac{\sqrt{2}}{3}\,\alpha_{i,\,i+2,\,i+4}&\sqrt{\frac{2}{3}}\,\alpha_{i,\,i+2,\,i+5}\\
\frac{1}{\sqrt{2}}\,\alpha_{i,\,i+1,\,i+3}&\frac{1}{\sqrt{6}}\,\alpha_{i,\,i+2,\,i+3}&0&\frac{1}{\sqrt{6}}\,\alpha_{i,\,i+3,\,i+4}&\frac{1}{\sqrt{2}}\,\alpha_{i,\,i+3,\,i+5}\\
\sqrt{\frac{2}{3}}\,\alpha_{i,\,i+1,\,i+4}&\frac{\sqrt{2}}{3}\,\alpha_{i,\,i+2,\,i+4}&\frac{1}{\sqrt{6}}\,\alpha_{i,\,i+3,\,i+4}&0&\sqrt{\frac{2}{3}}\,\alpha_{i,\,i+4,\,i+5}\\
\sqrt{2}\,\alpha_{i,\,i+1,\,i+5}&\sqrt{\frac{2}{3}}\,\alpha_{i,\,i+2,\,i+5}&\frac{1}{\sqrt{2}}\,\alpha_{i,\,i+3,\,i+5}&\sqrt{\frac{2}{3}}\,\alpha_{i,\,i+4,\,i+5}&0
\end{array}\right),
\end{align*}  
where $\alpha_{k_1,\,k_2,\,k_3}=\alpha_{\vert p}(e_{k_1},e_{k_2},e_{k_3})$ and the indices are cyclic of order $6$.

\smallskip

First we prove that the 'type' of $L_1^{i,\ell}(p)$ is constant for $p\in M$:
$L_1^{i,\ell}(p)$ is of one of the forms given in Theorem\,\ref{zus}. Since $L_1^{i,\ell}(p)$ depends continuously on $p\in M$ and $M$ is connected the form of $L_1^{i,\ell}(p)$
is constant for all $p\in M$.

\smallskip

Next we show that $L_1^{i,\ell}(p)$ can only be of the fifth or sixth type listed in Theorem\,\ref{zus} and thus coincide with the homogenous example.\\
Let us first suppose that $L_1^{6,0}(p)$ is of the first form listed in Theorem\,\ref{zus}. This implies $\alpha_{1,\,2,\,6}=\tfrac{5}{3\sqrt{2}}$ which in turn yields that the $(1,5)$-entry of $L_1^{1,0}(p)$ is given by $\tfrac{5}{3}$. However this coefficient does not arise in one of the possible $L_1$ listed in Theorem\,\ref{zus} and therefore this case cannot arise.\\
Next we suppose that $L_1^{6,0}(p)$ is of the second form listed in Theorem\,\ref{zus}. This implies $\alpha_{2,\,3,\,6}=1$ which in turn yields that the $(1,4)$-entry of $L_1^{2,0}(p)$ is given by $\sqrt{\tfrac{2}{3}}$. However this coefficient does not arise in one of the possible $L_1$ listed in Theorem\,\ref{zus} and therefore this case cannot arise.\\
Next we suppose that $L_1^{6,0}(p)$ is of the third form listed in Theorem\,\ref{zus}. This implies $\alpha_{1,\,3,\,6}=\sqrt{3}$ which in turn yields that the $(2,5)$-entry of $L_1^{1,0}(p)$ is given by $\sqrt{2}$. However this coefficient does not arise in one of the possible $L_1$ listed in Theorem\,\ref{zus} and therefore this case cannot arise.\\
Next we suppose that $L_1^{6,0}(p)$ is of the fourth form listed in Theorem\,\ref{zus}. This implies $\alpha_{1,\,2,\,6}=\alpha_{4,\,5,\,6}=1/\sqrt{2}$ and $\alpha_{1,\,4,\,6}=\alpha_{2,\,5,\,6}=\sqrt{2}$ which in turn yields that the $(1,5)$-entry and the $(3,5)$-entry of $L_1^{1,0}(p)$ are given by $1$. However this contradicts Theorem\,\ref{zus} and therefore this case cannot arise.\\
Finally for the fifth and sixth case one proves easily that everything is consistent and that in these cases all $\alpha(e_i,e_j,e_k)^2$ with $i,j,k\in\left\{1,...,6\right\}$ coincide with those of (\ref{ex2}) and (\ref{ex1}), respectively.

\smallskip

Therefore for all isoparametric hypersurfaces in $\Sph^7$ with $(g,m)=(6,1)$ all $\alpha(e_i,e_j,e_k)^2$ with $i,j,k\in\left\{1,...,6\right\}$ coincide with those of the homogeneous example and are in particular constant.
Hence the claim follows from Proposition\,\ref{ab}, i.e., Proposition\,12.5 in \cite{abresch2}.
\end{proof}

\appendix
\section{Counterexamples to the proof of Miyaoka \cite{mi1,mi3}}
We give counterexamples to some of Miyaoka's proofs in \cite{mi1,mi3}.

\subsection{Proposition\,8.1 and Proposition\,8.2 in \cite{mi3} are not compatible}
In Paragraph\,$3$ of \cite{mi3} Miyaoka claims to prove by contradiction that the case $\dim E=3$ does not occur. 
Although the statement is true the proof is incorrect: we show that Proposition\,8.1 and Proposition\,8.2 are not compatible.

\smallskip

In Proposition\,8.1 Miyaoka \cite{mi1,mi3} claims that $\lbrace e_3(t),X_1(t),X_2(t)\rbrace$ and
$\lbrace Z_1(t),Z_2(t)\rbrace$ 
constitute orthonormal frames of $E$ and $E^{\perp}$, respectively,
where
\begin{align*}
&X_1(t)=\alpha(t)(e_1(t)+e_5(t))+\beta(t)(e_2(t)+e_4(t)),\\
&X_2(t)=\tfrac{1}{\sqrt{\sigma(t)}}(\tfrac{\beta(t)}{\sqrt{3}}(e_1(t)-e_5(t))-\sqrt{3}\alpha(t)(e_2(t)-e_4(t)),\\
&Z_1(t)=\tfrac{1}{\sqrt{\sigma(t)}}(\sqrt{3}\alpha(t)(e_1(t)-e_5(t))+\tfrac{\beta(t)}{\sqrt{3}}(e_2(t)-e_4(t)),\\
&Z_2(t)=\beta(t)(e_1(t)+e_5(t))-\alpha(t)(e_2(t)+e_4(t))
\end{align*}
and $\alpha,\beta,\sigma$ are differentiable real functions on the interval $\lbrack 0,3\pi\rbrack$ satisfying $\alpha^2+\beta^2=\tfrac{1}{2}$ and $\sigma=2(3\alpha^2+\tfrac{1}{3}\beta^2)$.

\smallskip

Below we assume that $L(t)$ is given as in 
Lemma\,\ref{falle4nn} where we chose without loss of generality the $L_1$ with the $+$-sign.
Consider the following unit eigenvectors of $L(t)$:
\begin{align*}
e_1(t) = (f_1(t),\tfrac{1}{6}(3\sin(t)+\sin(2 t)), \tfrac{4}{9}\sin^2(t),\tfrac{1}{6}(3\sin(t) -\sin(2t)),f_1(t+\pi))^{tr},
\end{align*}
where $f_1(t)=\tfrac{1}{9}\cos^2(\tfrac{t}{2})(7+2\cos(t))$,
is a unit eigenvector of $L(t)$ with eigenvalue $\sqrt{3}$;
\begin{align*}
e_2(t) = (f_2(t),\cos^2(\tfrac{t}{2})(1-2\cos(t)),-\tfrac{2}{3}\sin(2t), -(1+ 2\cos(t))\sin^2(\tfrac{t}{2}),f_2(t+\pi))^{tr}
\end{align*}
where $f_2(t)=\tfrac{1}{6}(3+2\cos(t))\sin(t)$,
is a unit eigenvector of $L(t)$ with eigenvalue $\tfrac{1}{\sqrt{3}}$;
\begin{align*}
e_3(t) = (\tfrac{4}{9}\sin^2t, -\tfrac{2}{3}\sin(2t),\tfrac{1}{9}(1+8\cos(2t)),\tfrac{2}{3}\sin(2t), \tfrac{4}{9}\sin^2t)^{tr}
\end{align*}
is an eigenvector of $L(t)$ with eigenvalue $0$. Then $e_4(t)=\pm e_2(t+\pi)$ and $e_5(t)=\pm e_1(t+\pi)$ are eigenvectors of $L(t)$ with eigenvalues $-\tfrac{1}{\sqrt{3}}$
and $-\sqrt{3}$, respectively. Following Miyaoka \cite{mi3} we assume $e_4(t)=e_2(t+\pi)$ and $e_5(t)=e_1(t+\pi)$.
Thus we get $$E=\mbox{span}((0, 0, 1, 0, 0)^{tr}, (1, 0, 0, 0, 1)^{tr}, (0, -1, 0, 1, 0)^{tr}).$$
Therefore $e_1(t)+e_5(t),e_2(t)+e_4(t)\in E$ but $e_1(t)-e_5(t),e_2(t)-e_4(t)\in E^{\perp}$.
Consequently, the element $X_2(t)$ does not lie in $E$, contradicting Proposition\,8.1 in \cite{mi1}.

\smallskip

One may try to avoid this problem by another choice of the eigenvectors $e_4(t)$ and $e_5(t)$. Note that for any admissible choice of $e_4(t)$ and $e_5(t)$ we have:
if $\alpha(t)\neq 0$ and $\beta(t)\neq 0$ at least one of the vectors $X_1(t)$ or $X_2(t)$ does not lie in $E$.  Thus either $\alpha\equiv 0$ or $\beta\equiv 0$ and we may assume without loss of generality that $\alpha\equiv 0$. In order for $X_1(t),X_2(t)$ to lie in $E$ we must have $e_4(t)=-e_2(t+\pi)$ and $e_5(t)=e_1(t+\pi)$, which implies $e_4(0)=-e_2(\pi)$ and $e_4(\pi)=-e_2(0)$. However, this implies that the proof of Proposition\,8.2 \cite{mi1,mi3} does not work anymore. Indeed, we no longer obtain a proof by contradiction: just follow along the lines of this proof and use $e_1(\pi)=e_5(0)$, $e_2(\pi)=-e_4(0)$, $e_3(\pi)=e_3(0)$, $e_4(\pi)=-e_2(0)$ and $e_5(\pi)=e_1(0)$. 

\smallskip

Conclusion: the contradiction obtained in \cite{mi1} and \cite{mi3} results from the inadmissible
assumption that Proposition\,8.1 in \cite{mi1} and $e_4(t)=e_2(t+\pi)$, $e_5(t)=e_1(t+\pi)$ hold.
If we change the sign of exactly one of the eigenvectors $e_4(t)$ or $e_5(t)$ Proposition\,8.1 is true but then the proof of Proposition\,8.2
becomes incorrect.

\subsection{Counterexample to the proof of Proposition\,7.1 in \cite{mi3}}
In \cite{mi3} Proposition\,7.1 is used to exclude the case $\dim E=2$.

\smallskip

Below we suppose that $L(t)$ is given as in  
Lemma\,\ref{ph2}, where we assume that $L_1$ is of the first form stated in this lemma - the argument is similar for the case when $L_1$ is of the second form in that lemma.
Then
\begin{align*}
&e_1(t)= (\cos(t/2), 0, 0, 0, \sin(t/2))^{tr},\\
&e_2(t)= (0, \cos^2(t/2), \sin(t)/\sqrt{2},\sin^2(t/2), 0)^{tr},\\
&e_3(t)= (0, -\sin(t)/\sqrt{2}, \cos(t), \sin(t)/\sqrt{2}, 0)^{tr},\\
&e_4(t)= (0, \sin^2(t/2), -\sin(t)/\sqrt{2}, \cos^2(t/2), 0)^{tr},\\
&e_5(t)= (-\sin(t/2), 0, 0, 0, \cos(t/2))^{tr}
\end{align*}
constitutes an orthonormal basis of eigenvectors of $L(t)$ where the corresponding eigenvalues are given by $\sqrt{3},\tfrac{1}{\sqrt{3}},0,-\tfrac{1}{\sqrt{3}}$ and $-\sqrt{3}$, respectively.
Hence $e_3(\pi)=-e_3(0)$, $e_2(\pi)=e_4(0)$, $e_4(\pi)=e_2(0)$, $e_1(\pi)=e_5(0)$ and $e_5(\pi)=-e_1(0)$.
This example proves that not only the four cases listed in \cite{mi1,mi3}, namely $(e_1+e_5)(\pi)=(e_1+e_5)(0)$ and $(e_2+e_4)(\pi)=\pm(e_2+e_4)(0)$ or
 $(e_1+e_5)(\pi)=-(e_1+e_5)(0)$ and $(e_2+e_4)(\pi)=\pm(e_2+e_4)(0)$ occur. The missing cases cannot be excluded by the argument given in \cite{mi1,mi3}.

\section*{Acknowledgements}
It is a pleasure to thank Prof. Dr. U. Abresch for numerous mathematical discussions.
Furthermore, I am very grateful to Prof. Dr. W. Ziller for his many valuable comments. Finally, I would like to thank the referee for his very helpful recommendations.

\end{document}